\UseRawInputEncoding
\documentclass[12pt]{amsart}
\usepackage{setspace, amsmath, amsthm, amssymb, amsfonts, amscd, epic, graphicx, ulem, dsfont}
\usepackage[T1]{fontenc}
\usepackage{multirow}
\usepackage{bbm}
\usepackage{enumerate}

\makeatletter \@namedef{subjclassname@2010}{
  \textup{2020} Mathematics Subject Classification}
\makeatother

\newtheorem{thm}{Theorem}[section]

\newtheorem{cor}[thm]{Corollary}
\newtheorem{lem}[thm]{Lemma}
\newtheorem{pro}[thm]{Proposition}

\theoremstyle{remark}
\newtheorem*{rema}{Remark}

\theoremstyle{definition}
\newtheorem*{defn}{Definition}
\newtheorem{exa}[thm]{\textbf{Example}}

\newcommand{\Ima}{\operatorname{Im}}
\newcommand{\re}{\operatorname{Re}}
\newcommand{\tr}{\operatorname{tr}}

\newcommand{\R}{\mathbb{R}}

\newcommand{\N}{\mathbb{N}}

\newcommand{\C}{\mathbb{C}}

\begin{document}

\title[Nilpotence of Operators]{When Nilpotence Implies the Zeroness of Linear Operators}
\author[FRID et al.]{Nassima Frid, Mohammed Hichem Mortad$^*$ and Souheyb Dehimi}

\dedicatory{}
\thanks{* Corresponding author.}
\date{}


\address{(The first author): Département de
Mathématiques, Université Oran1, Ahmed Ben Bella, B.P. 1524, El
Menouar, Oran 31000, Algeria.}

\email{nassima.frid@yahoo.fr}

\address{(The corresponding and second author) Department of
Mathematics, Université Oran 1, Ahmed Ben Bella, B.P. 1524, El
Menouar, Oran 31000, Algeria.}

\email{mhmortad@gmail.com, mortad.hichem@univ-oran1.dz.}

\address{(The third author): Department of Mathematics, Faculty of Mathematics and Informatics,
University of Mohamed El Bachir El Ibrahimi, Bordj Bou Arréridj,
El-Anasser 34030, Algeria.}

\email{souheyb.dehimi@univ-bba.dz, sohayb20091@gmail.com}

\begin{abstract}
In this paper, we give conditions forcing nilpotent operators
(everywhere bounded or closed) to be null. More precisely, it is
mainly shown any closed or everywhere defined bounded nilpotent
operator with a positive (self-adjoint) real part is automatically
null.
\end{abstract}

\maketitle

\section*{Introduction}

First, we assume readers have some familiarity with the standard
notions and results in matrix and operator theories (see e.g.
\cite{Axler LINEAR ALG DONE RIGHT} and
\cite{Mortad-Oper-TH-BOOK-WSPC}), as well as unbounded operators
(see \cite{Weidmann} for the needed notions, cf.
\cite{Mortad-cex-BOOK}).

Let $H$ be a Hilbert space and let $B(H)$ be the algebra of all
bounded linear operators defined from $H$ into $H$. Recall that
$T\in B(H)$ is said to be positive, symbolically $T\geq0$, if
$<Tx,x>\geq0$ for all $x\in H$. Recall also that any $T$ may always
be expressed as $T=A+iB$ with $A,B\in B(H)$ being both self-adjoint
and $i=\sqrt{-1}$. Necessarily, $A=(T+T^*)/2$ which will be denoted
by $\re T$ and it is called the real part of $T$. Also,
$B=(T-T^*)/{2i}$ is the imaginary part of $T$, written $\Ima T$.

As is well known, nilpotence plays an important role in matrix
theory, and in operator theory in general. The following was shown
in \cite{Mortad-square-roots-normal}:

\begin{pro}\label{MORTAD-PROP}
If $T\in B(H)$ is such that $\re T\geq0$ and $T^2=0$, then $T=0$
\end{pro}

In this paper, we carry on this investigation and deal with the
general case.

We recall a few well established facts. For example, if $T\in B(H)$
is normal, then
\[\|T^n\|=\|T\|^n,~\forall n\in\N.\]
It seems noteworthy to emphasize that thanks to the previous
equality, if $T$ is nilpotent then "$T=0\Leftrightarrow T\text{ is
normal}$". Therefore, when we further assume that $\re T\geq0$ and
prove Theorem \ref{Main THM} below, then this will become yet
another characterization to be added to the 89 conditions equivalent
to the normality of a matrix already obtained in
\cite{Elsner-Ikramov} and \cite{Grone et al-normal matrices}. A
somehow related paper is \cite{Fillmore et al quasinilpotent}.

The second main topic of the paper deals with (unbounded) closed
operators. So let's recall briefly some notions about non
necessarily bounded operators.

If $S$ and $T$ are two linear operators with domains $D(S)$ and
$D(T)$ respectively, then $T$ is said to be an extension of $S$,
written as $S\subset T$, if $D(S)\subset D(T)$ and $S$ and $T$
coincide on $D(S)$.

The product $ST$ and the sum $S+T$ of two operators $S$ and $T$ are
defined in the usual fashion on the natural domains:

\[D(ST)=\{x\in D(T):~Tx\in D(S)\}\]
and
\[D(S+T)=D(S)\cap D(T).\]

When $\overline{D(T)}=H$, we say that $T$ is densely defined. In
such case, the adjoint $T^*$ exists and is unique. If $S\subset T$
and $S$ is densely defined, then $T$ too is densely defined and
$T^*\subset S^*$.

An operator $T$ is called closed if its graph is closed in $H\oplus
H$. If $T$ is densely defined, we say that $T$ is self-adjoint when
$T=T^*$; symmetric if $T\subset T^*$; normal if $T$ is
\textit{closed} and $TT^*=T^*T$. A symmetric operator $T$ is called
positive if
\[\langle Tx,x\rangle\geq 0, \forall x\in D(T).\]
Notice that unlike positive operators in $B(H)$, an unbounded
positive operator need not be self-adjoint.

In the event of the density of all of $D(S)$, $D(T)$ and $D(ST)$,
then
\[T^*S^*\subset (ST)^*,\]
with equality occurring when $S\in B(H)$. Also, when $S$, $T$ and
$S+T$ are densely defined, then
\[S^*+T^*\subset (S+T)^*,\]
and the equality holds if $S\in B(H)$.

The real and imaginary parts of a densely defined operator $T$ are
defined respectively by
\[\re T=\frac{T+T^*}{2} \text{ and } \Ima T=\frac{T-T^*}{2i}.\]
Clearly, if $T$ is closed, then $\re T$ is symmetric but it is not
always self-adjoint (it may even fail to be closed).

\begin{defn}\label{cartesian decomp unbd DEFN}(\cite{Ota-normal extensions-BPAS}) Let $T$ be a densely
defined operator with domain $D(T)\subset H$. If there exist densely
defined symmetric operators $A$ and $B$ with domains $D(A)$ and
$D(B)$ respectively and such that
\[T=A+iB \text{ with } D(A)=D(B),\]
then $T$ is said to have a Cartesian decomposition.
\end{defn}

\begin{rema}A densely defined operator $T$ admits a Cartesian decomposition if and only if $D(T)\subset D(T^*)$. In this case, $T=A+iB$ where
\[A=\re T \text{ and } B=\Ima T.\]
\end{rema}

\section{The Bounded Case}

The first result tells us that a (non-zero) operator $T$ with a
positive (or negative) real or imaginary part is never nilpotent. It
may be known to some readers especially when $\dim H<\infty$. The
proof when $\dim H=\infty$ here relies on the finite dimensional
case. In the next section, we generalize this result to closed
operators.

\begin{thm}\label{Main THM}
Let $T=A+iB\in B(H)$ and let $n\geq2$. If $T^n=0$ and $A\geq0$ (or
$B\geq 0$), then $T=0$.
\end{thm}

\begin{proof}The proof is carried out in two steps.
\begin{enumerate}
  \item Let $\dim H<\infty$. The proof uses a trace argument. First, assume that $A\geq0$.
Clearly, the nilpotence of  $T$ does yield $\tr T=0$. Hence
\[0=\tr(A+iB)=\tr A+i\tr B.\]

Since $A$ and $B$ are self-adjoint, we know that $\tr A,\tr B\in\R$.
By the above equation, this forces $\tr B=0$ and $\tr A=0$. The
positiveness of $A$ now intervenes to make $A=0$. Therefore, $T=iB$
and so $T$ is normal. Thus, and as alluded above,
\[0=\|T^n\|=\|T\|^n,\]
thereby, $T=0$.

In the event $B\geq0$, reason as above to obtain $T=A$ and so $T=0$,
as wished.
  \item Let $\dim H=\infty$. The condition $\re T\geq0$ is
  equivalent to $\re<Tx,x>\geq0$ for all $x\in H$. So if $E$ is a closed invariant
  subspace of $T$, then the previous condition also holds for
  $T|E:E\to E$.

  Now, we proceed to show that $T=0$, i.e. we must show that $Tx=0$ for all $x\in H$. So, let $x\in
  H$ and let $E$ be the span of $x, Tx,\cdots, T^{n-1}x$ (that is, the orbit of $x$ under the action of $T$). Hence $E$ is a finite dimensional subspace of $H$ (and so it is equally a Hilbert space).
  By the nilpotence assumption, we have
  \[T^nx=0,\]
  from which it follows that $E$ is invariant for $T$. So, by the first part
  of the proof (the finite dimensional case), we know that $T=0$ on
  $E$ whereby $Tx=0$. As this holds for any $x$, it follows that
  $T=0$ on $H$, as needed.
\end{enumerate}
\end{proof}

\begin{rema}For example, the condition $A\geq0$ may not just be
dropped. Indeed, if $T=\left(
                         \begin{array}{cc}
                           0 & 1 \\
                           0 & 0 \\
                         \end{array}
                       \right)$, then $T^2=0$ but $T\neq0$. Observe
                       finally
                       that
\[A=\re T=\frac{1}{2}\left(
                       \begin{array}{cc}
                         0 & 1 \\
                         1 & 0 \\
                       \end{array}
                     \right)
\]
is neither positive nor negative for $\sigma(A)=\{-1/2,1/2\}$.
\end{rema}

\begin{rema}
As mentioned above, the power of Theorem \ref{Main THM} lies in the
fact it easily allows us to test the non-nilpotence of a given
operator. For example, let $V$ be the Volterra's operator defined on
$L^2(0,1)$, that is,
\[Vf(x)=\int_0^xf(t)dt,~f\in L^2(0,1).\]
Then, it well known that $V$ is not nilpotent. Let's corroborate
this fact using Theorem \ref{Main THM}. Since $\re V\geq0$ (see e.g.
Exercise 9.3.21 in \cite{Mortad-Oper-TH-BOOK-WSPC}), assuming the
nilpotence of $V$ would make $V=0$, and this is impossible. Thus,
$V$ is not nilpotent.

The previous example also tells that the assumption may not be
weakened to quasinilpotence (recall that quasinilpotence means that
spectrum is reduced to the singleton $\{0\})$.
\end{rema}

Here is an alternative reformulation of Theorem \ref{Main THM} over
finite dimensional spaces.

\begin{cor}
Let $T\in M_n(\C)$ be nilpotent (with $T\neq0$). Then $(T+T^*)/2$
(or $(T-T^*)/{2i}$) has at least two eigenvalues of opposite signs.
\end{cor}

\section{The Unbounded Case}

We confine our attention now to the case of unbounded nilpotent
operators. We choose to use \^{O}ta's definition in
\cite{Ota-nilpotent-idempotent} of nilpotence (notice that S.
\^{O}ta gave the definition in the case $n=2$).

\begin{defn}
Let $T$ be a non necessarily bounded operator with a dense domain
$D(T)$. We say that $T$ is nilpotent if $T^n$ is well defined and
\[T^n=0 \text{ on }D(T)\]
for some $n\in\N$ (hence $D(T^n)=D(T^{n-1})=\cdots D(T)$).
\end{defn}

Thanks to the following lemma, there are not any unbounded
self-adjoint \textit{nilpotent} operators!

\begin{lem}\label{ghghhgghghghghghghghghghghghghghghgh}(\cite{Sebestyen-Stochel-JMAA}) If $H$ and $K$ are two
Hilbert spaces and if $T:D(T)\subset H\to K$ is a densely defined
closed operator, then
\[D(T)=D(T^*T)\Longleftrightarrow T\in B(H,K).\]
\end{lem}

Since for a normal $T$,  $D(T^*T)=D(T^2)$ we see that there are not
any unbounded normal \textit{nilpotent} operators either. So, it is
natural to ask is whether there are unbounded closed symmetric
unbounded operators? The answer is still negative! In fact, any
densely defined closed nilpotent operator $T$ with $D(T)\subset
D(T^*)$ is everywhere bounded.

\begin{pro}\label{wqssdfhgkytmpeutaa}
Let $T$ be a densely defined closed nilpotent operator with domain
$D(T)$ such that $D(T)\subset D(T^*)\subset H$. Then $T\in B(H)$.

In particular, if $T$ is a closed densely defined nilpotent
symmetric or hyponormal operator, then $T=0$ everywhere on $H$.
\end{pro}

\begin{proof}Let $T$ be a densely defined closed  operator
with domain $D(T)\subset H$ such that $T^n=0$ on $D(T)$ for some $n$
and $D(T)\subset D(T^*)$. It is seen that
\[D(T)=D(T^2)\subset D(T^*T)\subset D(T)\]
whereby $D(T^*T)=D(T)$. Since $T$ is closed, Lemma
\ref{ghghhgghghghghghghghghghghghghghghgh} yields $T\in B(H)$.

The last statement of the proposition follows from the general
theory. Indeed, when $T\in B(H)$, then $T$ is self-adjoint iff it is
symmetric. Accordingly, $T=0$ since $T^n=0$ everywhere on $H$. The
case of hyponormality is also known to readers.
\end{proof}

\begin{rema}The previous result may be reformulated as: Any closed densely defined nilpotent
operator having a Cartesian decomposition is necessarily everywhere
bounded.
\end{rema}

As alluded above, Theorem \ref{Main THM} remains valid in the
context of closed operators.

\begin{thm}Let $T=A+iB$ where either $A$ or $B$ is positive with $D(T)\subset D(T^*)$. If $T$ is nilpotent, then
$T\in B(H)$ is normal thereby $T=0$ everywhere on $H$.
\end{thm}

\begin{proof}
Since $T^n=0$ on $D(T)$ for some natural integer $n$, we have
$D(T^2)=D(T)$. Reason as in Proposition \ref{wqssdfhgkytmpeutaa} to
obtain $T\in B(H)$. We have thus gone back to the setting of Theorem
\ref{Main THM}, i.e. we obtain $T=0$, as wished.
\end{proof}

Before stating and proving the last result in this paper, we give
some auxiliary results which are also interesting in their own.
Notice that they might well be known to specialists, however, they
are not documented (to the best of our knowledge).

It is worth noticing in passing that there are unbounded
self-adjoint operators $A$ and $B$ such that $A+iB\subset 0$ (where
0 designates the zero operator on all of $H$), yet $A\not\subset 0$
and $B\not\subset 0$. For example, let $A$ and $B$ be unbounded
self-adjoint operators such that $D(A)\cap D(B)=\{0_H\}$ (see e.g.
\cite{KOS}). Assuming $D(A)=D(B)$ makes the whole difference.
Indeed:

\begin{pro}\label{kkkkkkkkkkkkkkkkkkkkkkkkkkk}
Let $A$ and $B$ be two densely defined symmetric operators with
domains $D(A),D(B)\subset H$ respectively. Assume that $D(A)=D(B)$.
If $A+iB\subset 0$, then $A\subset 0$ and $B\subset 0$. If $A$ (or
$B$) is further taken to be closed, then $A=B=0$ everywhere on $H$.
\end{pro}

\begin{proof}
By assumption $A+iB\subset 0$. Since $D(A)=D(B)$, it ensues that
$A=-iB$. But $A$ and $B$ are both symmetric, and so the only
possible outcome is $A\subset 0$ and $B\subset 0$.

Since $A\subset 0$, it follows that $A^*=0$ everywhere on $H$. By
the closedness of $A$, we obtain $A=0$. A similar reasoning applies
to $B$ because $A=-iB$ makes $B$ closed and so $B=0$ as well.
\end{proof}

It is known that pointwise commutativity of unbounded (self-adjoint
or normal) operators does not always mean their strong
commutativity, i.e. the commutativity of their spectral measures
(witness Nelson's counterexample). So, the next result on (strong)
commutativity might be unknown to some readers.

\begin{pro}\label{cvbncvbncvbncvbncvbncvbn}(Cf. \cite{MHM1})
Let $A$ and $B$ be two unbounded self-adjoint operators with domains
$D(A)$ and $D(B)$ respectively. Assume that $A$ is also positive and
that $D(A)=D(B)$. If $BA\subset AB$, then $A$ commutes strongly with
$B$.
\end{pro}

\begin{proof}By hypothesis, $BA\subset AB$. Hence $B(A+I)\subset
(A+I)B$ because
\[D[B(A+I)]\subset D[(A+I)B].\]

Since $A$ is self-adjoint and positive, it results that $A+I$ is
boundedly invertible. Left and right multiplying by $(A+I)^{-1}$
yield
\[(A+I)^{-1}B\subset B(A+I)^{-1}.\]
This means that $A$ commutes strongly with $B$, therefore completing
the proof.
\end{proof}

As readers are aware, the condition $D(T^2)=D(T)$ is strong. Why not
call a densely defined operator $T$ nilpotent when $T^n\subset 0$
for a certain $n$? The main issue would be that it is quite
conceivable to have $T^n$ defined only at 0. See e.g. \cite{CH},
\cite{Dehimi-Mortad-CHERNOFF}, \cite{Mortad-TRIVIALITY POWERS
DOMAINS}, \cite{Mortad-cex-BOOK}, and
\cite{SCHMUDG-1983-An-trivial-domain} (cf. \cite{Arlinski-Zagrebnov}
and \cite{Brasche-Neidhardt}). A recent somewhat related paper
\cite{Dehimi-Mortad-squares-polynomials} might be of some interest
to readers.

Let us treat this case anyway.

\begin{thm}\label{T2 subset 0 THM}
Let $T=A+iB$ where $A$ and $B$ are self-adjoint (one of them is also
positive), $D(A)=D(B)$ and $D(BA)\subset D(AB)$. If $T^2\subset 0$,
then $T\in B(H)$ is normal, and so $T=0$ everywhere on $H$.
\end{thm}

\begin{proof}
Assume that $A$ is positive (the proof in the case of the
positiveness of $B$ is similar). Let $T=A+iB$. Clearly,
\[A^2-B^2+i(AB+BA)\subset (A+iB)A+i(A+iB)B=T^2\subset 0.\]
Since $D(A)=D(B)$, it follows that
\[D(A^2)=\{x\in D(A):Ax\in D(A)\}=\{x\in D(A):Ax\in D(B)\}=D(BA).\]
In a similar manner, it is seen that $D(B^2)=D(AB)$. Thus,
\[D(A^2-B^2)=D(AB+BA).\]
We also have $D(BA)\subset D(AB)$. Accordingly,
\[D(A^2-B^2)=D(AB+BA)=D(BA)=D(A^2).\]

Since $A$ is self-adjoint, so is $A^2$ and in particular $A^2$ is
necessarily densely defined. Thus, $A^2-B^2$ and $AB+BA$ are both
densely defined. Now, by the symmetricity (only) of both $A$ and $B$
we have that both $AB+BA$ and $A^2-B^2$ are symmetric. By
Proposition \ref{kkkkkkkkkkkkkkkkkkkkkkkkkkk}, we get $AB+BA\subset
0$. Hence $BA\subset -AB$ (for $D(BA)\subset D(AB)$) and so
\[BA^2\subset -ABA\subset A^2B.\]

As $A$ is positive, we obtain $BA\subset AB$ by say \cite{Bernau
JAusMS-1968-square root}. By Proposition
\ref{cvbncvbncvbncvbncvbncvbn}, we have that $A$ commutes strongly
with $B$. Whence $T$ is normal. Hence $T^2$ too is normal and so by
maximality $T^2\subset 0$ becomes $T^2=0$ everywhere on $H$.
Consequently, $D(T^2)=H$ and hence $D(T)=H$. Since $T$ is closed, it
follows by the closed graph theorem that $T\in B(H)$. Finally, $T=0$
follows by the normality of $T$, as needed.
\end{proof}

\begin{rema}
Another way of obtaining $A=B=0$ in the result above (and without
using Proposition \ref{cvbncvbncvbncvbncvbncvbn}) reads: Since $B$
is self-adjoint, $B=U|B|=|B|U$ where $U$ is unitary and
self-adjoint, i.e. $U^2=I$ and $U^*=U$ (see e.g.
\cite{Gustafson-Mortad-II} or \cite{Jung-Mortad-Stochel}).

As above, we may obtain $BA\subset -AB$ and $A^2-B^2\subset 0$.
Since $D(BA)=D(A^2)$ and $D(AB)=D(B^2)$, we get $A^2\subset B^2$.
Since $A^2$ and $B^2$ are both self-adjoint, a maximality argument
yields $A^2=B^2$ which, upon passing to the unique positive square
root, implies that $A=|B|$ as $A$ is positive. Hence $B=UA=AU$.
Therefore
\[UA^2=U|B|A=BA\subset -AB=-UA^2.\]
Hence
\[U^2A^2=A^2\subset -U^2A^2=-A^2.\]
Thus, $A^2=-A^2$ and so $A=0$, thereby $B=0$ as well.
\end{rema}

A variant of Theorem \ref{T2 subset 0 THM} is:

\begin{thm}
Let $A$ and $B$ be two self-adjoint operators (one of them is also
positive) such that $D(A^2)=D(B^2)$. If $T=A+iB$ and $T^2\subset 0$,
then $T=0$ everywhere on $H$.
\end{thm}

\begin{proof}We need only go back to the assumptions of Theorem
\ref{T2 subset 0 THM}. WLOG, assume that $A$ is positive. Since $B$
is self-adjoint, $B^2$ is self-adjoint and positive. Then, by
Theorem 9.4 in \cite{Weidmann}, we obtain
\[D(A)=D(\sqrt{A^2})=D(\sqrt{B^2})=D(|B|)=D(B)\]
by invoking the closedness of $B$ and the positiveness of $A$. Hence
\[D(BA)=D(A^2)=D(B^2)=D(AB).\]
Therefore, we have recovered all of the assumptions of Theorem
\ref{T2 subset 0 THM}. The remaining parts  of the proof stay
unchanged.
\end{proof}

We finish with an example showing the importance of the
self-adjointness of $A$ and $B$.

\begin{exa}There is a densely defined closed symmetric positive operator $A$ such that
$T:=A+iA$ obeys $T^2\subset 0$ yet $T\not\subset 0$.

To obtain such example, recall that P. R. Chernoff \cite{CH}
obtained a densely defined closed, symmetric and positive operator
$A$ such that $D(A^2)=\{0\}$. Now, let $B=A$ and set
$T=A+iA=(1+i)A$. Then
\[D(T^2)=D(A^2)=\{0\}\]
and so $T^2\subset 0$ trivially. Observe in the end that
$T\not\subset 0$, i.e. $T$ does not vanish on $D(T)$. Observe in the
end that neither $A$ nor $B$ were self-adjoint.
\end{exa}

\section*{Acknowledgements}

The corresponding author wishes to express his gratitude to
Professor Alexander M. Davie (the University of Edinburgh) for his
help in the infinite dimensional case of the proof of Theorem
\ref{Main THM}.

Notice that the idea of using the polar decomposition in the remark
below the proof of Theorem \ref{T2 subset 0 THM}, comes from one of
the readers of the paper. This was suggested when the assumption
$D(BA)=D(AB)$ was made instead of $D(BA)\subset D(AB)$ in an earlier
version of that theorem. Then, we saw the improvement using the
assumption $D(BA)\subset D(AB)$.

\end{document}